\documentclass[12pt,english]{article}
\usepackage[margin=1in]{geometry}
\usepackage{amsmath,amsthm,amssymb}
\usepackage{tikz}
\usepackage{graphicx}
\usepackage{hyperref}

\newtheorem{theorem}{Theorem}
\newtheorem{proposition}[theorem]{Proposition}
\newtheorem{lemma}[theorem]{Lemma}
\newtheorem{corollary}[theorem]{Corollary}
\theoremstyle{remark}
\newtheorem{example}[theorem]{Example}
\newtheorem{remark}[theorem]{Remark}

\numberwithin{theorem}{section}
\numberwithin{equation}{section}

\DeclareMathOperator{\Z}{\mathbb{Z}}

\DeclareMathOperator{\N}{\mathbb{N}}

\DeclareMathOperator{\dg}{dg}
\DeclareMathOperator{\Des}{Des}
\DeclareMathOperator{\Asc}{Asc}
\DeclareMathOperator{\maj}{maj}
\DeclareMathOperator{\Inv}{Inv}
\DeclareMathOperator{\inv}{inv}
\DeclareMathOperator{\coinv}{coinv}
\newcommand{\hsigma}{\hat{\sigma}}
\newcommand{\naf}{\mathrm{NAF}}
\newcommand{\hdg}{\hat{\dg}}
\renewcommand{\hat}{\widehat}

\begin{document}

\title{Combinatorial proof of a permuted basement Macdonald polynomial identity}
\author{Daniel Orr and Johnny Rivera, Jr.}
\date{\today}
\maketitle

\begin{abstract}
A well-known and fundamental property of the Macdonald polynomials $P_\lambda(x;q,t)$ is their invariance under the transformation sending $(q,t)$ to $(q^{-1},t^{-1})$. Recently, Concha and Lapointe~\cite{CL} showed that this property extends in an interesting, nontrivial way to an identity for partially symmetric Macdonald polynomials. Their identity played a key role in the work \cite{BWO} linking partially symmetric Macdonald polynomials to parabolic flag Hilbert schemes. In this paper, we refine the Concha-Lapointe identity to a sub-family of Alexandersson's permuted basement Macdonald polynomials \cite{A} and give a combinatorial proof of the refined identity. We show also that the Concha-Lapointe identity is equivalent to the assertion that (normalized) partially symmetric Macdonald polynomials are fixed under the Kazhdan-Lusztig involution.
\end{abstract}

\section{Introduction}

The Macdonald polynomials\footnote{See \cite[Chapter VI]{M}, \cite[Chapter 3]{C:book}, and \cite{HHL} for background on the theory of Macdonald polynomials which is most useful for this paper.} $P_\lambda(x;q,t)$ \cite{M}---which are symmetric polynomials in variables $x=(x_1,\dotsc,x_n)$ depending rationally on parameters $q,t$---satisfy a number of beautiful identities. One of these is their invariance under inversion of the parameters: $P_\lambda(x;q^{-1},t^{-1})=P_\lambda(x;q,t)$.

In this paper, we consider a larger family of Macdonald polynomials obtained by relaxing the requirement of symmmetry in the variables $x_1,\dotsc,x_n$. Namely, for arbitrary $m$ with $0\le m\le n$, we consider polynomials which are symmetric in $x_1,\dotsc,x_m$, with no condition of symmetry imposed upon $x_{m+1},\dotsc,x_n$. These are the partially symmetric Macdonald polynomials $P_{\lambda|\gamma}(x;q,t)$ studied in \cite{G,L}, where $\lambda$ is a partition (having at most $m$ parts) and $\gamma$ is a composition (of length $n-m$). The $P_{\lambda|\gamma}$ interpolate between the symmetric $P_\lambda$ (when $m=n$) and the non-symmetric Macdonald polynomials $E_\gamma$ \cite{C} (when $m=0$).

One motivation for studying the partially symmetric polynomials $P_{\lambda|\gamma}$ is that many features of the $P_\lambda$ admit nontrivial extensions to the partially symmetric setting. For instance, Lapointe \cite{L} formulated a positivity conjecture for the $P_{\lambda|\gamma}$ extending the well-known Macdonald positivity theorem proved by Haiman \cite{H}. Subsequent work of Concha-Lapointe \cite{CL} established a number of fundamental properties of the $P_{\lambda|\gamma}$, including the following extension of the identity $P_\lambda(x;q,t)=P_\lambda(x;q^{-1},t^{-1})$.

\begin{theorem}[\cite{CL}]\label{CL theorem}
For any partition $\lambda$ with at most $m$ parts and $\gamma=(\gamma_1,\dotsc,\gamma_{n-m})\in(\Z_{\ge 0})^{n-m}$, we have
\begin{equation}\label{CL identity}
P_{\lambda|\gamma}(x_1,\dotsc,x_m,qx_{n},\dotsc,qx_{m+1};q^{-1},t^{-1})
= q^{\gamma_1+\dotsm+\gamma_{n-m}}t^{\mathrm{inv}(\gamma)-\ell(\omega_0^{[m+1,n]})}T_{\omega_0^{[m+1,n]}}P_{\lambda|\gamma},
\end{equation}
where $\omega_0^{[m+1,n]}\in S_{[m+1,n]}$ is the longest element of the symmetric group of the set $[m+1,n]=\{m+1,m+2,\dotsc,n\}$, $T_{\omega_0^{[m+1,n]}}$ is the corresponding Demazure-Lusztig operator (see \S\ref{MacDL}), and $\mathrm{inv}(\gamma)=|\{i<j : \gamma_i>\gamma_j\}|$.
\end{theorem}

We observe that \eqref{CL identity} reduces to the aforementioned identity when $m=n$, since in that case $\gamma$ is empty. Moreover, when $m=0$, it produces a known identity for the non-symmetric Macdonald polynomials; see \cite[Prop. 3.3.3]{C:book} or \cite[Prop. 6]{A} for the non-symmetric identity and also \cite[Theorem 4.8]{I} for its connection to the Kazhdan-Lusztig involution. The shift of variables $x_{m+1},\dotsc,x_n$ by the parameter $q$ is an interesting feature of the partially symmetric setting. This shift is not present in the symmetric case, while in the non-symmetric case it cancels with $q^{\gamma_1+\dotsm+\gamma_{m-n}}$ by homogeneity.

The general case of \eqref{CL identity} played a key role in the work \cite{BWO} of Bechtloff Weising and the first named author, where it was used to show that an explicit modified form of the $P_{\lambda|\mu}$ is sent, under a canonical isomorphism due Carlsson-Gorsky-Mellit \cite{CGM}, to the basis of torus fixed points for the equivariant $K$-theory of parabolic flag Hilbert schemes. (The precise statement proved in \cite{BWO} was formulated as a conjecture by Goodberry and the first named author in \cite{GO}.) In this context, the analog of Macdonald positivity \cite{H} remains open; see \cite{K,L} for positivity conjectures and \cite{BHMPS} for very recent progress involving newly introduced non-symmetric LLT polynomials.

The proof of \eqref{CL identity} in \cite{CL} is algebraic and employs a characterization of the $P_{\lambda|\gamma}$ as eigenfunctions of certain difference operators. Our original aim in this work was to give a combinatorial proof of \eqref{CL identity} based on the Haglund-Haiman-Loehr formula \cite{HHL}, thereby explaining in combinatorial terms the peculiarities of \eqref{CL identity} required by the partially symmetric setting. The combinatorics which we explore requires carefully manipulating statistics defined on fillings of cell-based diagrams that we associate to $\lambda|\gamma$. In the course of working out the details, we found that most natural statement involves an equality of summands from each side of \eqref{CL identity} which are instances of Alexandersson's permuted basement polynomials \cite{A}. Our main result, Theorem \ref{complement thm}, thus gives a refinement of \eqref{CL identity} to these summands.

We also give the following interpretation of \eqref{CL identity} in terms of the Kazhdan-Lusztig involution~$*$, extending \cite[Theorem 4.8]{I} and providing yet another proof of the Concha-Lapointe identity.

\begin{theorem}\label{KL theorem}
Under the Kazhdan-Lusztig involution $*$ on $\mathbb{Q}(q,t)[x_1,\dotsc,x_n]^{S_m}$ (see \eqref{KL eM}),
\begin{equation}\label{KL P}
P_{\lambda|\gamma}^* = t^{\mathrm{inv}(\lambda|\gamma)}P_{\lambda|\gamma}
\end{equation}
for all $\lambda|\gamma$ as in Theorem~\ref{CL theorem} above, where $\mathrm{inv}(\lambda|\gamma)=\mathrm{inv}(\gamma)+|\{(i,j)\in[m]\times[n-m]:\lambda_i>\gamma_j\}|$. Furthermore, equations \eqref{CL identity} and \eqref{KL P} are equivalent.
\end{theorem}

\section*{Acknowledgements}

We thank Milo Bechtloff Weising and Olya Mandelshtam for helpful discussions. D.O. gratefully acknowledges support from the Simons Foundation.

\section{Macdonald polynomials}

\subsection{Notation and conventions}

For any set $A$, let $S_A$ be the symmetric group of permutations of $A$, with the group operation given by composition: $\pi\rho=\pi\circ\rho$. For all $a,b\in\mathbb{N}$, let $[a,b]=\{i\in\mathbb{N} : a\le i\le b\}$, and let $\omega_0^{[a,b]}$ denote the longest element of $S_{[a,b]}$ given by $\omega_0^{[a,b]}(i) = a+b-i$ for all $i\in [a,b]$; when $a>b$, the set $[a,b]$ is empty, the group $S_{[a,b]}$ is trivial, and $\omega_0^{[a,b]}$ is the identity element. For all $n\in\mathbb{N}$, we write $[n]=[1,n]=\{1,\dotsc,n\}$, $S_n = S_{[n]}$, and $\omega_0^{[n]} = \omega_0^{[1,n]}$. 

For arbitrary $n\in \mathbb{N}$, let $\mathbb{Q}(q,t)[x_1,\dotsc,x_n]$ be the algebra of polynomials in $x_1,\dotsc,x_n$ over the field $\mathbb{Q}(q,t)$ of rational functions in indeterminates $q$ and $t$. The symmetric group $S_n$ acts naturally from the left on $\mathbb{Q}(q,t)[x_1,\dotsc,x_n]$ by algebra automorphisms, which are specified by $\pi(x_i) = x_{\pi(i)}$ for all $\pi\in S_n$ and $i\in [n]$. For any subgroup $G\subset S_n$, let $\mathbb{Q}(q,t)[x_1,\dotsc,x_n]^G$ the subalgebra of polynomials which are invariant under $G$. The group $S_n$ also acts naturally from the left on $\mathbb{Z}^n$ by the linear automorphisms $\pi(\mu)=(\mu_{\pi^{-1}(1)},\dotsc,\mu_{\pi^{-1}(n)})$ for all $\pi\in S_n$ and all $\mu=(\mu_1,\dotsc,\mu_n)\in\mathbb{Z}^n$.

\subsection{Symmetric and non-symmetric Macdonald polynomials}\label{MacDL}

For an element $\lambda\in(\mathbb{Z}_{\ge 0})^n$ such that $\lambda_1\ge \dotsm\ge \lambda_n$, i.e., a partition with at most $n$ parts, denote by $P_\lambda=P_\lambda(x_1,\dotsc,x_n;q,t)$ the (symmetric) \textit{Macdonald polynomial} of \cite{M}. These form a $\mathbb{Q}(q,t)$-basis of the space $\mathbb{Q}(q,t)[x_1,\dotsc,x_n]^{S_n}$ of symmetric polynomials.

For a composition $\mu\in (\mathbb{Z}_{\ge 0})^n$, denote by $E_\mu=E_\mu(x_1,\dotsc,x_n;q,t)$ the (type $GL_n$) \textit{non-symmetric Macdonald polynomial} in the conventions of \cite{HHL}. These form a $\mathbb{Q}(q,t)$-basis of the space $\mathbb{Q}(q,t)[x_1,\dotsc,x_n]$ of all polynomials.


For $i=1,\dotsc,n-1$, let $\sigma_i=(i,i+1)\in S_n$ be the simple transposition and let $T_i$ be the \textit{Demazure-Lusztig} operator on $\mathbb{Q}(q,t)[x_1,\dotsc,x_n]$ given by
\begin{equation*}
T_i f = t\sigma_i(f)+(t-1)x_{i+1}\frac{\sigma_i(f)-f}{x_i-x_{i+1}}.
\end{equation*}
For $\pi\in S_n$ and any reduced expression $\pi=\sigma_{i_1}\dotsc \sigma_{i_\ell}$, where $\ell=\ell(\pi)$ is the length of $\pi$, it is well-known that the operator $T_\pi$ defined by
\begin{equation*}
T_\pi=T_{i_1}\dotsm T_{i_\ell}
\end{equation*}
does not depend on the choice of reduced expression.

The relationship between symmetric and non-symmetric Macdonald polynomials can then be stated as follows. By \cite[(3.3.14)]{C:book}, for a partition $\lambda$ with at most $n$ parts, one has
\begin{equation*}
P_\lambda = \sum_{\pi\in S_n^\lambda} T_\pi E_\lambda,
\end{equation*}
where $S_n^\lambda\subset S_n$ is the set of minimal coset representatives for $S_n/(S_n)_\lambda$, and $(S_n)_\lambda$ is the stabilizer of $\lambda$. We note also, by \cite[(3.3.40)]{C:book}, that whenever $\sigma_i(\mu)=\mu$, one has
\begin{equation}\label{Estab}
T_iE_\mu=tE_\mu.
\end{equation}

\subsection{Partially symmetric and permuted basement Macdonald polynomials}

Let $m$ be an integer such that $0\le m\le n$. Regard $S_m=S_{[m]}$ as the subgroup of $S_n$ fixing $m+1,\dotsc,n$. For $\lambda\in(\mathbb{Z}_{\ge 0})^m$ a partition with at most $m$ parts and $\gamma\in(\mathbb{Z}_{\ge 0})^{n-m}$ a composition, the \textit{partially symmetric Macdonald polynomial} $P_{\lambda|\gamma}$ of \cite{G,L} is defined---following the the conventions of \cite{G}---by
\begin{equation}\label{psPdef}
P_{\lambda|\gamma} = \sum_{\pi\in S_m^\lambda} T_\pi E_{\mu},
\end{equation}
where $\mu=(\lambda,\gamma)\in(\mathbb{Z}_{\ge 0})^n$. These form a $\mathbb{Q}(q,t)$-basis of $\mathbb{Q}(q,t)[x_1,\dotsc,x_n]^{S_m}$, the algebra of polynomials which are symmetric in $x_1,\dotsc,x_m$.


Our refinement of Theorem \ref{CL theorem} concerns the summands of \eqref{psPdef}. For any $\pi\in S_n$ and $\mu\in(\mathbb{Z}_{\ge 0})^n$, define
\begin{equation*}
E_\mu^\pi = t^{-\ell_\mu(\pi)}T_\pi E_\mu,
\end{equation*}
where $\ell_\mu(\pi)$ is the number of pairs $(i,j)\in[n]^2$ such that $i<j$, $\pi(i) >\pi(j)$, and $\mu_i\le\mu_j$. In the next section, we explain how the $E_\mu^\pi$ connect to the combinatorially-defined \textit{permuted basement Macdonald polynomials} based on diagram fillings, with $\mu$ specifying the shape and $\pi$ the basement of the filling; in particular, see Remark~\ref{Remark on A}.

\section{Combinatorics}

In this section, we recall the combinatorial formula of Haglund, Haiman, and Loehr \cite{HHL} for non-symmetric Macdonald polynomials and its extension to permuted basement Macdonald polynomials due to Alexandersson \cite{A}. Most of the following definitions come directly from \cite{HHL} or from \cite{A} (adapted to the conventions of \cite{HHL}). We note that one difference arises, namely: when we define the inversion number $\inv(\hsigma)$ of an augmented filling $\hsigma$, we also take into account the permuted basements that we are allowing. In Haglund, Haiman, and Loehr's work, the authors assume that the basement of their diagrams contains the increasing permutation. The slight adjustment in our definition of the inversion number takes into account that our basements need not be the increasing permutation and ensures that the inversion and coninversion numbers coincide with the number of inversion and coinversion triples, respectively, which we define later.

\subsection{Column Diagrams} 

We may visualize a composition $\mu=(\mu_1,\dotsc,\mu_n)\in(\mathbb{Z}_{\ge 0})^n$ as a diagram of $n$ columns with $\mu_i$ boxes in column $i$. The \textit{column diagram} of $\mu$ is the set $$\dg'(\mu)=\{(i,j)\in(\Z_{\geq0})^2:1\leq i\leq n, 1\leq j\leq\mu_i\}$$ in Cartesian coordinates. That is to say that $i$ indexes the columns and $j$ indexes the rows of the column diagram of $\mu$. Further, the \textit{augmented diagram} of $\mu$ is given by $$\hat{\dg}(\mu)=\dg'(\mu)\cup\{(i,0):1\leq i\leq n\},$$ where we simply adjoin $n$ boxes to row 0. For $\mu\in(\mathbb{Z}_{\ge 0})^n$ and $u=(i,j)\in\dg'(\mu)$, we define
\begin{align*}
    \text{leg}(u)&=\{(i,j')\in\dg'(\mu):j'>j\},\\
    \text{arm}^{\text{left}}(u)&=\{(i',j)\in\dg'(\mu):i'<i,\mu_{i'}\leq\mu_i\},\\
    \text{arm}^{\text{right}}(u)&=\{(i',j-1)\in\hat{\dg}(\mu):i'>i,\mu_{i'}<\mu_i\},\\
    \text{arm}(u)&=\text{arm}^{\text{left}}(u)\cup\text{arm}^{\text{right}}(u),
\end{align*}
and we set 
\begin{align*}
    l(u)&=|\text{leg}(u)|=\mu_i-j,\\
    a(u)&=|\text{arm}(u)|.
\end{align*}

\begin{example}\label{diagram example}
    The diagram below depicts the diagram $\mu=(4,3,3,2,3,4)$ and the box $u =(3,2)$ with the cells belonging to $\text{leg}(u)$, $\text{arm}^{\text{left}}(u)$, and $\text{arm}^{\text{right}}(u)$ marked by $x$, $y$, and $z$, respectively. Here, $l(u)=1$ and $a(u)=2$.

    
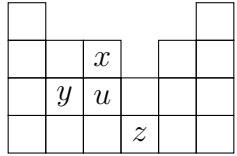
\begin{center}
\begin{tikzpicture}[scale=0.5]
  \begin{scope}
    \foreach \x/\val in {0/1,1/3,2/2,3/6,4/4,5/5} {
      \draw (\x,1) rectangle ++(1,-1);
    }
    \foreach \x/\val in {0/1,1/3,2/2,3/6,4/4,5/5} {
      \draw (\x,2) rectangle ++(1,-1);
    }
    \foreach \x/\val in {0/1,1/3,2/2} {
      \draw (\x,3) rectangle ++(1,-1);
    }
    \draw (4,3) rectangle ++(1,-1);
    \draw (5,3) rectangle ++(1,-1);
    \draw (0,4) rectangle ++(1,-1);
    \draw (5,4) rectangle ++(1,-1);
    \node at (2+0.5,1.5) {$u$};
    \node at (2+0.5,2.5) {$x$};
    \node at (1+0.5,1.5) {$y$};
    \node at (3+0.5,0.5) {$z$};
        
  \end{scope}
\end{tikzpicture}
\end{center}

\end{example}

\subsection{Fillings and Statistics}
A \textit{filling} of $\mu$ is a function $$\sigma:\dg'(\mu)\rightarrow [n],$$ where $[n]=\{1,2,\dotsc,n\}$. For a fixed permutation $\pi\in S_n$, the associated \textit{augmented filling} is the associated filling $\hat{\sigma}:\hat{\dg}(\mu)\rightarrow[n]$ of the augmented diagram such that $\hat{\sigma}$ agrees with $\sigma$ on $\dg'(\mu)$, and $\hat{\sigma}((j,0))=\pi(j)$ agrees with the given permutation for all $j\in[n]$. We refer to the permutation $\pi$, or, equivalently row 0 of $\hsigma$ as the \textit{basement} of the augmented filling. Basements given by arbitrary permutations were considered in \cite{A}, but here they are adapted to the conventions of \cite{HHL}. We refer the reader to Example \ref{filling example} below for an illustration of our conventions.

Boxes $u,v\in(\mathbb{Z}_{\ge 0})^2$ such that $u\neq v$ are said to \textit{attack} each other if either
\begin{enumerate}
    \item they are in the same row (they have the form $(i,j), (i',j))$, or
    \item they are in consecutive rows, and the box in the lower row is to the right of the box in the upper row (they have the form $(i,j), (i',j-1)$ where $i'>i$).
\end{enumerate} A filling $\hat{\sigma}:\hat{\dg}(\mu)\rightarrow[n]$ is non-attacking if $\hat{\sigma}(u)\neq\hat{\sigma}(v)$ for each pair of attacking boxes $u,v\in\hat{\dg}(\mu)$. We say that $\sigma:\dg'(\mu)\rightarrow[n]$ is non-attacking if its associated augmented filling $\hat{\sigma}$ is non-attacking. Following \cite{A}, for a composition $\mu$ and a permutation $\pi$, we denote the set of all non-attacking augmented fillings of $\mu$ with basement $\pi$ by $\naf_\mu^\pi$. If $\pi$ is omitted, we assume that it is the identity permutation.

We write $d(u)=(i,j-1)$ to denote the box directly below the box $u=(i,j)$. A \textit{descent} of an augmented filling $\hat{\sigma}:\hat{\dg}(\mu)\rightarrow[n]$ is a box $u\in \dg'(\mu)$ such that $\hat{\sigma}(u)>\hat{\sigma}(d(u))$. Similarly, a box $u\in \dg'(\mu)$ is an \textit{ascent} if $\hsigma(u)<\hsigma(d(u)).$ For such a filling, we define
\begin{align*}
    \Des(\hsigma)&=\{u\in\dg'(\mu) : \text{$u$ is a descent of $\hat{\sigma}$}\},\\
    \Asc(\hsigma)&=\{u\in\dg'(\mu) : \text{$u$ is an ascent of $\hat{\sigma}$}\},\\
    \maj(\hat{\sigma})&=\sum_{u\in\Des(\hat{\sigma})}(l(u)+1).
\end{align*}
The \textit{reading order} on the set $\hat{\dg}(\mu)$ is the total ordering of the boxes in $\hat{\dg}(\mu)$ row by row, from top to bottom, and from right to left. In symbols, $(i,j)<(i',j')$ if $j<j'$ or if $j=j'$ and $i>i'$. An \textit{inversion} of a filling $\hat{\sigma}:\hat{\dg}(\mu)\rightarrow[n]$ is a pair of boxes $u,v\in \hat{\dg}(\mu)$ such that $u,v$ attack each other, $u$ comes before $v$ in the reading order, and $\hat{\sigma}(u)>\hat{\sigma}(v).$ We define $\Inv(\hat{\sigma})$ to be the set of all inversions and 
\begin{align*}
    \inv(\hat{\sigma})&=|\Inv(\hat{\sigma})|-|\{i<j:\text{$\mu_i\leq\mu_j$ and $\hat{\sigma}(i,0)<\hat{\sigma}(j,0)$}\}|-\sum_{u\in\Des(\hat{\sigma})}a(u),\\\coinv(\hat{\sigma})&=\left(\sum_{u\in\dg'(\mu)}a(u)\right)-\inv(\hat{\sigma}).
\end{align*} 

\begin{example}\label{filling example}
    The diagram below depicts an augmented filling $\hsigma\in\naf_{(3,2,2,1,2,3)}^{132645}$, where the permutation $\pi=132645$ is written in one-line notation (so that $\pi(4)=6$ for instance).
\begin{center}
\begin{tikzpicture}[scale=0.5]
  \begin{scope}
    \foreach \x/\val in {0/1,1/3,2/2,3/6,4/4,5/5} {
      \draw (\x,1) rectangle ++(1,-1);
        \node at (\x+0.5,0.5) {\val};
    }
    \foreach \x/\val in {0/1,1/3,2/2,3/6,4/4,5/5} {
      \draw (\x,2) rectangle ++(1,-1);
        \node at (\x+0.5,1.5) {\val};
    }
    \foreach \x/\val in {0/1,1/3,2/2} {
      \draw (\x,3) rectangle ++(1,-1);
        \node at (\x+0.5,2.5) {\val};
    }
    \draw (4,3) rectangle ++(1,-1);
    \node at (4+0.5,2.5) {\textbf{\textcolor{red}{6}}};
    \draw (5,3) rectangle ++(1,-1);
    \node at (5+0.5,2.5) {\textbf{\textcolor{blue}{4}}};
    \draw (0,4) rectangle ++(1,-1);
    \node at (0+0.5,3.5) {\textbf{\textcolor{red}{5}}};
    \draw (5,4) rectangle ++(1,-1);
    \node at (5+0.5,3.5) {\textbf{\textcolor{blue}{1}}};
    \draw[black, ultra thick] (0,1) -- (6,1);
  \end{scope}
\end{tikzpicture}
\end{center}
\end{example}

Recall that the bottom row, the row below the bolded black line, is row 0. The descent set $\Des(\hsigma)$ consists of the two boxes bolded and colored in \textbf{\textcolor{red}{red}}, and the ascent set $\Asc(\hsigma)$ consists of the two boxes bolded and colored in \textbf{\textcolor{blue}{blue}}. We find that $\maj(\hsigma)=2$. In row 2 we have 8 inversions, and in rows 1 and 0 we have 12 inversions each. Between consecutive rows, there are a total of 8 inversions, so $|\Inv(\hsigma)|=40$, which we encourage the reader to verify. Then we find that $\inv(\hsigma)=40-6-5=29$, and $\coinv(\hsigma)=32-29=3$.

\subsection{Combinatorial formula}

For the identity permutation $\pi=12\dotsm n$, the main result of \cite{HHL} is the combinatorial formula for the non-symmetric Macdonald polynomials given below. In \cite{A}, this was extended to basements given by arbitrary permutations. 

\begin{theorem}[{\cite{A,HHL}}]
\label{comb form thm}
    The permuted basement Macdonald polynomials are given by the formula
    \begin{equation*}
    E_\mu^\pi(x;q,t)=\sum_{\sigma\in\naf_\mu^\pi}x^\sigma q^{\maj(\hat{\sigma})}t^{\coinv(\hat{\sigma)}}\prod_{\substack{u\in\dg'(\mu)\\\hat{\sigma}(u)\neq\hat{\sigma}(d(u))}}\frac{1-t}{1-q^{l(u)+1}t^{a(u)+1}},
    \end{equation*}
    where $x^\sigma=\prod_{u\in\dg'(\mu)}x_{\sigma(u)}.$
\end{theorem}

\begin{remark}\label{Remark on A}
Denoting by $A^\pi_\mu$ the permuted basement non-symmetric Macdonald polynomials as defined in \cite{A}, one has the precise relationship $E^\pi_\mu=A^{\pi\omega_0}_{\omega_0(\mu)}$, where $\omega_0=\omega_0^{[n]}\in S_n$. This can be shown by induction on $\ell(\pi)$ using \cite[Prop. 15]{A}. We note that, in the expression $A^{\pi\omega_0}_{\omega_0(\mu)}$, $\pi\omega_0$ denotes group operation of composition on $S_n$ and $\omega_0(\mu)=(\mu_n,\dotsc,\mu_1)$.  
\end{remark}

It is useful to reformulate $\inv(\hat{\sigma})$ and $\coinv(\hat{\sigma})$ as the number of \textit{inversion triples} and \textit{coinversion triples}, respectively. We say that a \textit{triple} consists of three boxes $u,v,w\in\hat{\dg}(\mu)$ such that $$w=d(u)\text{ and }v\in\text{arm}(u).$$

With this description, we can characterize triples pictorially. Figure \ref{triple picture} depicts Type I and Type II triples, assuming that the column to which $u$ belongs is strictly taller than the column to which $v$ belongs in the Type I case and weakly taller in the Type II case.

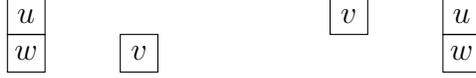
\begin{figure}[h]
\centering
\begin{tikzpicture}[scale=.5]
  \begin{scope}
    \draw (0,1) rectangle ++(1,-1);
        \node at (0+0.5,0.5) {$w$};
        \draw (3,1) rectangle ++(1,-1);
        \node at (3+0.5,0.5) {$v$};
        \draw (0,2) rectangle ++(1,-1);
        \node at (0+0.5,1.5) {$u$};
  \end{scope}
\end{tikzpicture}
\hspace{2cm}
\begin{tikzpicture}[scale=.5]
  \begin{scope}
    \draw (0,2) rectangle ++(1,-1);
        \node at (0+0.5,1.5) {$v$};
        \draw (3,1) rectangle ++(1,-1);
        \node at (3+0.5,0.5) {$w$};
        \draw (3,2) rectangle ++(1,-1);
        \node at (3+0.5,1.5) {$u$};
  \end{scope}
\end{tikzpicture}
\caption{Type I (left) and Type II (right) triples.}
\label{triple picture}
\end{figure}

Given a filling $\hat{\sigma}$ and boxes $x,y$ in $\hat{\dg}(\mu)$, with $x<y$ in the reading order, set
\begin{equation*}
    \chi_{xy}(\hsigma)=\begin{cases}
        1 & \text{if } \hat{\sigma}(x)>\hsigma(y),\\
        0 & \text{otherwise}.
    \end{cases}
\end{equation*}
If $(u,v,w)$ is a triple, then $v$ attacks $u$ and $w$, so we have $\chi_{uv}(\hsigma)=1$ if and only if $(u,v)\in\inv(\hsigma)$. Similarly, we have $\chi_{vw}(\hsigma)=1$ if and only if $(v,w)\in\inv(\hsigma)$, and we have $\chi_{uw}(\hsigma)=1$ if and only if $u\in\Des(\hsigma)$. It then follows that 
\begin{equation*}
    \chi_{uv}(\hsigma)+\chi_{vw}(\hsigma)-\chi_{uw}(\hsigma)\in\{0,1\}.
\end{equation*} With that said, we say that a triple $(u,v,w)$ is an \textit{inversion triple} of $\hsigma$ if $\chi_{uv}(\hsigma)+\chi_{vw}(\hsigma)-\chi_{uw}(\hsigma)=1$, and we say that $(u,v,w)$ is a \textit{coinversion triple} otherwise. We have the following useful lemma regarding triples.
\begin{lemma}[{\cite[Lemma 3.6.1]{HHL}}]
\label{attacking boxes}
    Every pair of attacking boxes in $\hat{\dg}(\mu)$ occurs as either $\{u,v\}$ or $\{v,w\}$ in a unique triple $(u,v,w)$, with the exception that an attacking pair $\{(i,0), (i',0)\}$ in row $0$ such that $i<i'$ and $\mu_i\leq\mu_{i'}$ does not belong to any triple.
\end{lemma}

We now state and prove a proposition from \cite{HHL} since the proof differs slightly on account of our definition of $\inv(\hsigma)$.

\begin{proposition}[\cite{HHL}, Proposition 3.6.2]
    Let $\sigma$ be a filling of $\mu$. The number of inversion triples (respectively, coinversion triples) of $\hsigma$ is equal to $\inv(\hsigma)$ (respectively, $\coinv(\hsigma))$.
\end{proposition}

\begin{proof}
    We claim that the number of inversion triples is given by the sum over all triples
    \begin{equation}
        \sum_{(u,v,w)}(\chi_{uv}(\hsigma)+\chi_{vw}(\hsigma)-\chi_{uw}(\hsigma)).
    \end{equation} Naively, one might assume that the contribution to the sum from the first two terms comes from $|\Inv(\hsigma)|$; however, it follows from Lemma \ref{attacking boxes} that we must subtract from $|\Inv(\hsigma)|$ the number of inversions in the basement. Hence, the contribution to the sum from the first two terms is given by $|\Inv(\hsigma)|-|\{i<j:\mu_i\leq\mu_j, \hat{\sigma}(i,0)<\hat{\sigma}(j,0) \}|$. Further, the contribution from the last term is simply given by $-\sum_{u\in\Des(\hsigma)}a(u)$, which proves that $\inv(\hsigma)$ is equal to the number of inversion triples. Further, we have exactly one triple $(u,v,w)$ for each $u\in\dg'(\mu)$ and $v\in\text{arm}(u)$. Thus, the number of triples is equal to $\sum_{u\in\dg'(\mu)}\text{arm}(u)$, so the number of coinversion triples is equal to $(\sum_{u\in\dg'(\mu)}\text{arm}(u))-\inv(\hsigma)=\coinv(\hsigma)$, and we are done.
\end{proof}

We have the following useful description for coinversion triples: 

\begin{lemma}[\cite{HHL}, Lemma 3.6.3]
\label{coinv structure}
    If $\sigma$ is a non-attacking filling, then a triple $(u,v,w)$ is a coinversion triple if and only if $\hsigma(u)<\hsigma(v)<\hsigma(v)$ or $\hsigma(v)<\hsigma(w)<\hsigma(u)$ or $\hsigma(w)<\hsigma(u)<\hsigma(v)$.
\end{lemma}

In other words, a triple is a coinversion triple if and only if its entries increase clockwise in Type I or counterclockwise in Type II. Knowing this will be helpful in our later proofs. In \cite{HHL} the authors recast their combinatorial formula as a formula for what they call the ``opposite'' Macdonald polynomials $E_\mu(x; q^{-1}, t^{-1})$. They define
\begin{align*}
    \maj'(\hsigma)&=(\sum_{\substack{u\in\dg'(\mu)\\\hsigma(u)\neq\hsigma (d(u))}}(l(u)+1))-\maj(\hsigma)=\sum_{u\in\Asc(\hsigma)}(l(u)+1),\\
    \coinv'(\hsigma)&=(\sum_{\substack{u\in\dg'(\mu)\\\hsigma(u)\neq\hsigma (d(u))}}(a(u)))-\coinv(\hsigma),
\end{align*}
where it is helpful to realize that $\maj'(\hsigma)=\maj(\hsigma')$ and $\coinv'(\hsigma)=\coinv(\hsigma')$, where $\hsigma'(u)=n+1-\hsigma(u)$. In the following section, this realization will prove to be especially useful. We end this section with the following corollary of Theorem \ref{comb form thm}.

\begin{corollary}[\cite{HHL}, Corollary 3.6.4]\label{hhl corollary}
    For all $\mu\in (\mathbb{Z}_{\ge 0})^n$,
        \begin{equation*}
            E_\mu(x;q^{-1},t^{-1})=\sum_{\sigma\in\naf_\mu^\pi}x^\sigma q^{\maj'(\hat{\sigma})}t^{\coinv'(\hat{\sigma})}\prod_{\substack{u\in\dg'(\mu)\\\hat{\sigma}(u)\neq\hat{\sigma}(d(u))}}\frac{1-t}{1-q^{l(u)+1}t^{a(u)+1}}.
        \end{equation*}
\end{corollary}

Theorem \ref{complement thm} below is our main result regarding the permuted basement Macdonald polynomials, which we ultimately use to prove the Concha-Lapointe identity stated in Theorem~\ref{CL theorem}. In the statements of both Corollary~\ref{hhl corollary} and Theorem~\ref{complement thm}, notice the variable substitutions $q\mapsto q^{-1}$ and $t\mapsto t^{-1}$. Corollary \ref{hhl corollary} provides us with a combinatorial interpretation of how these substitutions interact with the non-attacking fillings. Specifically, the substitution $q\mapsto q^{-1}$ indicates that the variable $q$ now tracks the statistic $\maj'(\hsigma)$, and the substitution $t\mapsto t^{-1}$ indicates that the variable $t$ now tracks the statistic $\coinv'(\hsigma)$.

\subsection{Permuted basement identity}

As noted in \cite{A}, Corollary \ref{hhl corollary} provides an interesting relationship between non-attacking fillings of a diagram with the identity permutation in the basement, and non-attacking fillings of a diagram with the reversed identity permutation in the basement. The following work generalizes this result, and provides a similar relationship between non-attacking fillings with complementary basements.

For a permutation $\pi\in S_n$, we write $\pi$ in the usual one-line notation $\pi=\pi(1)\pi(2)\dotsm\pi(n)$. Its complement, denoted $\pi^c$, is given by $\pi^c=(n+1-\pi(1))(n+1-\pi(2))\dotsm(n+1-\pi(n))$. For example, if $\pi=1432$, then $\pi^c=4123$. We also consider those permutations that can be written as the concatenation of two permutations of disjoint sets of consecutive integers. For example, the permutation $\pi=132546$ can be thought of as the concatenation of $\pi_1=132$ and $\pi_2=546$. For our purposes, we consider the complements of $\pi_1$ and $\pi_2$ with respect to the sets on which they are permutations. That is, if $A=[a,b]=\{a,a+1,\dotsc,b-1,b\}$ is a finite set of consecutive positive integers, where $1\leq a\leq b$, then for any $\pi\in S_A$, we define $$\pi^c(i)=b+a-\pi(i), \quad\text{for all $i\in [a,b]$}.$$ (In case $a>b$, we define $\pi^c$ to be the identity element, i.e., the empty word in one-line notation.) Then with this definition of the complement, we see that for $\pi_1=132$ and $\pi_2=546$, we have $\pi_1^c=312$ and $\pi_2^c=564$. We now state and work towards proving the following, which is our main result and the promised refinement of the Concha-Lapointe identity (cf. \S\ref{CL section}).

\begin{theorem}
\label{complement thm}
    Suppose $\pi_1\in S_m$ and $\pi_2 \in S_{[m+1,n]}$ for some $n\in\N$ and integer $m$ such that $0\le m \le n$. Let $\pi=\pi_1\pi_2$ be the concatenation of these permutations. Let $\mu\in \mathbb{N}^n$ and set $p=\sum_{m+1\leq i\leq n}\mu_i$. Then
    \begin{equation}
    \label{generalized equation}
        E_\mu^{\pi_1^c\pi_2^c}(x_1,\dotsc,x_n;q,t)=q^{-p}E_\mu^{\pi_1\pi_2}(x_m,\dotsc,x_1,qx_n,\dotsc,qx_{m+1};q^{-1},t^{-1}).
    \end{equation}
\end{theorem}


Before proving Theorem \ref{complement thm}, we prove three lemmas to help us understand the underlying combinatorics behind the stated identity. We provide a pictorial example demonstrating the proof of each lemma.

\begin{lemma}
    Let $\mu\in\N^n$ and let $\pi_1\in S_m$ and $\pi_2\in S_{[m+1,n]}$ for some $n\in \mathbb{N}$ and integer $m$ such that $0\le m\le n$. Then
    \begin{equation*}
        |\naf_\mu^{\pi_1\pi_2}|=|\naf_\mu^{\pi_1^c\pi_2^c}|.
    \end{equation*}
\end{lemma}

\begin{proof}
    There exists a natural function $f:\naf_\mu^{\pi_1\pi_2}\rightarrow\naf_\mu^{\pi_1^c\pi_2^c}$ such that for any $\hsigma\in\naf_\mu^{\pi_1\pi_2}$, $f(\hsigma)$ is the filling defined on each box of $\dg'(\mu)$ by $$f(\hsigma(u))=\begin{cases}\pi_1^c\circ\pi_1^{-1}\circ\hsigma(u) & \text{if } \hsigma(u)\leq m,\\
    \pi_2^c\circ\pi_2^{-1}\circ\hsigma(u)&\text{if } \hsigma(u)\geq m+1,\end{cases}$$ for each $u\in\dg'(\mu)$. Indeed, $f$ is well-defined. If we take a box $(i,0)$ in the basement of $\hsigma$ such that $\hsigma((i,0))\leq m$, we must have $i\in[m]$. Then we have
    \begin{align*}
        f(\hsigma((i,0)))&=\pi_1^c\circ\pi_1^{-1}\circ\hsigma((i,0))\\
        &=\pi_1^c\circ\pi_1^{-1}\circ\pi(i)\\
        &=\pi_1^c(i).
    \end{align*} Similarly, by taking a box $(i,0)$ in the basement of $\hsigma$ such that $\hsigma((i,0))>m$, we must have $i\in [m+1,n]$. Then we have
    \begin{align*}
        f(\hsigma((i,0)))&=\pi_2^c\circ\pi_2^{-1}\circ\hsigma((i,0))\\
        &=\pi_2^c\circ\pi_2^{-1}\circ\pi(i)\\
        &=\pi_2^c(i),
    \end{align*}
    and for any two boxes $u,v\in \dg'(\mu)$, we have $\hsigma(u)=\hsigma(v)$ if and only if $f(\hsigma(u))=f(\hsigma(v))$, and since $\hsigma$ is non-attacking, $f(\hsigma)$ is also non-attacking. Hence, $f(\hsigma)\in\naf_\mu^{\pi_1^c\pi_2^c}.$ This function is clearly one-to-one and onto, so $f$ is a bijection, which proves the statement.
\end{proof}

\begin{example}
    The diagram below depicts the augmented filling $\hsigma$ of a non-attacking filling $\sigma$ of $\mu=(3,2,2,1,2,3)$ and the corresponding filling $f(\hsigma)$.
\begin{center}
\begin{tikzpicture}[scale=0.5]
  \begin{scope}
     \node at (-12.5,2) {$\hsigma=$};
    \foreach \x/\val in {-11/1,-10/3,-9/2,-8/5,-7/4,-6/6} {
      \draw (\x,1) rectangle ++(1,-1);
        \node at (\x+0.5,0.5) {\val};
    }
    \foreach \x/\val in {-11/1,-10/3,-9/2,-8/5,-7/4,-6/6} {
      \draw (\x,2) rectangle ++(1,-1);
        \node at (\x+0.5,1.5) {\val};
    }
    \foreach \x/\val in {-11/1,-10/3,-9/2,-7/5,-6/6} {
      \draw (\x,3) rectangle ++(1,-1);
        \node at (\x+0.5,2.5) {\val};
    }
    \foreach \x/\val in {-11/4,-6/3} {
      \draw (\x,4) rectangle ++(1,-1);
        \node at (\x+0.5,3.5) {\val};
    }
    \node at (-4.5,2) {,};
  
    \foreach \x/\val in {0/3,1/1,2/2,3/5,4/6,5/4} {
      \draw (\x,1) rectangle ++(1,-1);
        \node at (\x+0.5,0.5) {\val};
    }
    \foreach \x/\val in {0/3,1/1,2/2,3/5,4/6,5/4} {
      \draw (\x,2) rectangle ++(1,-1);
        \node at (\x+0.5,1.5) {\val};
    }
    \foreach \x/\val in {0/3,1/1,2/2,4/5,5/4} {
      \draw (\x,3) rectangle ++(1,-1);
        \node at (\x+0.5,2.5) {\val};
    }
    \foreach \x/\val in {0/6,5/1} {
      \draw (\x,4) rectangle ++(1,-1);
        \node at (\x+0.5,3.5) {\val};
    }
    \node at (-2,2) {$f(\hsigma)=$};
    \draw[black, ultra thick] (-11,1) -- (-5,1);
    \draw[black, ultra thick] (0,1) -- (6,1);
  \end{scope}
\end{tikzpicture}
\end{center}
Notice that $\hsigma\in\naf_\mu^{132546}$ and $f(\hsigma)\in\naf_\mu^{312564}.$
\end{example}

\begin{lemma}
\label{maj lemma}
    Suppose $\pi_1\in S_m$ and $\pi_2 \in S_{[m+1,n]}$ for some $n\in\N$ and integer $m$ such that $0\le m \le n$. For a given composition $\mu\in\N^n$, let $\hsigma\in\naf_\mu^{\pi_1\pi_2}$ be arbitrary. Further, let $L=|(f( \hsigma))^{-1}([m+1,n])|$, and let $p=\sum_{m+1\leq i\leq n}\mu_i$. Then
    \begin{equation*}
        \label{maj,maj'}\maj(f(\hsigma))=\maj'(\sigma)-p+L.
    \end{equation*}
\end{lemma}

\begin{proof}
    Informally speaking, $L$ represents the number of boxes of $\dg'(\mu)$ that have fillings belonging to the set $[m+1,n]$ under the filling $f(\hsigma)$. We show that both sides of $(\ref{maj,maj'})$ count the same objects. First, we introduce a relation $\sim$ on the boxes of $\hdg(\mu)$ under the filling $\hsigma$. For any $u,v\in\hdg(\mu)$, we say that $u\sim v$ if both $\hsigma(u)$, $\hsigma(v)\in[m]$ or both $\hsigma(u)$, $\hsigma(v)\in[m+1,n]$.  By definition, $$\maj(f(\hsigma))=\sum_{u\in\Des(f(\hsigma))}(l(u)+1),$$ which we may rewrite as $$\maj(f(\hsigma))=\sum_{\substack{u\in \Des(f(\hsigma))\\u\sim d(u)}}(l(u)+1)+\sum_{\substack{u\in \Des(f(\hsigma))\\u\nsim d(u)}}(l(u)+1).$$ Notice that by definition of $f$, we have that $u\in\Des(f(\hsigma))$ and $u\sim d(u)$ if and only if $u\in\Asc(\hsigma)$ and $u\sim d(u)$. Similarly, $u\in \Des(f(\hsigma))$ and $u\nsim d(u)$ if and only if $u\in\Des(\hsigma)$ and $u\nsim d(u)$. Hence, we may write $$\maj(f(\hsigma))=\sum_{\substack{u\in \Asc(\hsigma)\\u\sim d(u)}}(l(u)+1)+\sum_{\substack{u\in \Des(f(\hsigma))\\u\nsim d(u)}}(l(u)+1),$$ which we finally rewrite as 
    \begin{equation}
        \label{maj of fhsig}
        \maj(f(\hsigma))=\sum_{u\in\Asc(\hat{\sigma})}(l(u)+1)-\sum_{\substack{u\in \Asc(\hat{\sigma})\\\hat{\sigma}(u)\nsim\hat{\sigma}(d(u))}}(l(u)+1)+\sum_{\substack{u\in \Des(\hat{\sigma})\\\hat{\sigma}(u)\nsim\hat{\sigma}(d(u))}}(l(u)+1).
    \end{equation} Notice that the first term of the right-hand side of (\ref{maj of fhsig}) is given by $\maj'(\hsigma)$ by definition. Additionally, notice that $-p+L$ yields the number of boxes in the first $m$ columns with fillings in $[m+1,n]$ and the negative number of boxes in the last $n-m$ columns with fillings in $[m]$. We show next that the last two terms of $(\ref{maj of fhsig})$ also count these boxes, which proves $(\ref{maj,maj'})$.

    Begin by choosing any column in $\hdg(\mu)$, and consider the highest box in this column. We initiate a process by which we read this column from top to bottom, either adding or subtracting $l(u)+1$, which at any point in this process can be thought of as the total number of boxes we have seen so far, depending on one of two conditions:
    \begin{enumerate}
        \item[(1)] If $\hsigma(u)\in[m]$ and $\hsigma(d(u))\in[m+1,n]$, subtract $l(u)+1.$ 
        \item[(2)] If $\hsigma(u)\in[m+1,n]$ and $\hsigma(d(u))\in[m]$, add $l(u)+1$.
    \end{enumerate}
    Otherwise, do nothing, and continue reading down the column. We are concerned with what the final sum of this process represents combinatorially. Observe that condition (1) is equivalent to the second term in $(\ref{maj of fhsig})$ and condition (2) is equivalent to the third term in $(\ref{maj of fhsig})$. Since we only add $l(u)+1$ when $\hsigma(u)\in[m+1,n]$ and $\hsigma(d(u))\in[m]$, the partial sum we have after adding indicates the total number of boxes $u$ in our column we have seen so far such that $\hsigma(u)\in[m+1,n]$. On the other hand, since we only subtract when $\hsigma(u)\in[m]$ and $\hsigma(d(u))\in[m+1,n]$, the partial sum we have after subtracting indicates the negative number of boxes $u$ in our column we have seen so far such that $\hsigma(u)\in[m]$.

    If the column we choose is any of the first $m$ columns, the last thing we do in this process is add, if we do anything at all, since the basement in the first $m$ columns is filled with entries in $[m]$. If we do nothing at all during this process, the final sum we have is $0$, indicating there are no boxes $u$ within the column such that $\hsigma(u)\in[m+1,n]$ and $\hsigma(d(u))\in[m]$. On the other hand, if the column we choose is any of the last $n-m$ columns, the last thing we do in this process is subtract, if we do anything at all, since the basement in the last $n-m$ columns is filled with entries in $[m+1,n]$. If we do nothing at all during this process, the final sum we have is $0$, indicating there are no boxes $u$ within the column such that $\hsigma(u)\in[m]$ and $\hsigma(d(u))\in[m+1]$. Therefore, completing this process for every column in $\hdg(\mu)$ counts the total number of boxes in the first $m$ columns with fillings in $[m+1,n]$ and the negative number of boxes in the last $n-m$ columns with fillings in $[m]$, which is what we wanted to show.   
\end{proof}

\begin{example}
For $\mu=(5,2,2,1,2,5)$, consider the the augmented filling of $\sigma\in\naf_\mu^{312564}$, where $\pi_1=312$ and $\pi_2=564$, depicted below.

\begin{center}
\begin{tikzpicture}[scale=0.5]
  \begin{scope}
    \foreach \x/\val in {0/3,1/1,2/2,3/5,4/6,5/4} {
      \draw (\x,1) rectangle ++(1,-1);
        \node at (\x+0.5,0.5) {\val};
    }
    \foreach \x/\val in {0/3,1/1,2/2,3/5,4/6,5/4} {
      \draw (\x,2) rectangle ++(1,-1);
        \node at (\x+0.5,1.5) {\val};
    }
    \foreach \x/\val in {0/3,1/1,2/2,4/6,5/4} {
      \draw (\x,3) rectangle ++(1,-1);
        \node at (\x+0.5,2.5) {\val};
    }
    \foreach \x/\val in {0/5,5/3} {
      \draw (\x,4) rectangle ++(1,-1);
        \node at (\x+0.5,3.5) {\val};
    }
    \foreach \x/\val in {0/2,5/1} {
      \draw (\x,5) rectangle ++(1,-1);
        \node at (\x+0.5,4.5) {\val};
    }
    \foreach \x/\val in {0/6,5/5} {
      \draw (\x,6) rectangle ++(1,-1);
        \node at (\x+0.5,5.5) {\val};
    }
    \draw[->] (0,5.4) to[out=180, in=180] node[left] {\(+1\)} (0,4.6);
    \draw[->] (0,4.4) to[out=180, in=180] node[left] {\(-2\)} (0,3.6);
    \draw[->] (0,3.4) to[out=180, in=180] node[left] {\(+3\)} (0,2.6);

    \draw[->] (6,5.4) to[out=0, in=0] node[right] {\(+1\)} (6,4.6);
    \draw[->] (6,3.4) to[out=0, in=0] node[right] {\(-3\)} (6,2.6);
    \draw[black, ultra thick] (0,1) -- (6,1);
  \end{scope}
\end{tikzpicture}
\end{center}

On either side of the diagram, we start from highest box in the column and either add $l(u)+1$ when $\hsigma(u)\in\{4,5,6\}$ and $\hsigma(d(u))\in[3]$, or subtract $l(u)+1$ when $\hsigma(u)\in[3]$ and $\hsigma(d(u))\in\{4,5,6\}$. On the left in the first column, the total sum of this process is $2$, which represents the total number of boxes in this column with fillings from $\{4,5,6\}$. On the right in the last column, the total sum of this process is $-2$, which represents the negative number of boxes in this column with fillings in $[3]$.
\end{example}

\begin{lemma}
\label{coinv lemma}
    Suppose $\pi_1\in S_m$ and $\pi_2 \in S_{[m+1,n]}$ for some $n\in\N$ and integer $m$ such that $0\le m\le n$. Let $\hsigma\in\naf_\mu^{\pi_1\pi_2}$ be arbitrary. Then
    \begin{equation*}
        \coinv(f(\hsigma))=\coinv'(\hsigma).
    \end{equation*}
\end{lemma}

\begin{proof}
    Recall that $\coinv(f(\hsigma))$ counts the number of coinversion triples of $f(\hsigma)$. Further, by Lemma \ref{coinv structure}, coinversion triples contain distinct entries. Without loss of generality, we consider only coniversion triples of Type I as a similar argument holds for coinversion triples of Type II. In Type I, the entries of a coninversion triple must increase clockwise. Thus, without loss of generality, we may assume that $f(\hsigma(u))<f(\hsigma(v))<f(\hsigma(w))$ for a coinversion triple $(u,v,w)$ of $f(\hsigma)$. If $f(\hsigma(u))\sim f(\hsigma(v))\sim f(\hsigma(w))$, then we have $\hsigma(u)>\hsigma(v)>\hsigma(w)$, so $(u,v,w)$ is an inversion triple with distinct entries in $\hsigma$. If $f(\hsigma(u))\nsim f(\hsigma(v))\sim f(\hsigma(w))$, then we must have $\hsigma(v)>\hsigma(w)>\hsigma(u)$, so $(u,v,w)$ is an inversion triple with distinct entries in $\hsigma$. Lastly, if $f(\hsigma(u))\sim f(\hsigma(v))\nsim f(\hsigma(w))$, the we must have $\hsigma(w)>\hsigma(u)>\hsigma(v)$, so $(u,v,w)$ is an inversion triple with distinct entries in $\hsigma$. Similarly, an inversion triple of $\hsigma$ with distinct entries always corresponds to a coinversion triple of $f(\hsigma)$. Hence, $\coinv(f(\hsigma))=\coinv'(\hsigma)$. 
\end{proof}

\begin{example}
    Consider the two diagrams below, where the boxes in $f(\hsigma)$ that are bolded and \textbf{\textcolor{blue}{blue}} form a coinversion triple of Type I, and the corresponding boxes in $\hsigma$ that are bolded and \textbf{\textcolor{blue}{blue}} form an inversion triple of Type I with distinct entries.
    
\begin{center}
\begin{tikzpicture}[scale=0.5]
  \begin{scope}
     \node at (-12.5,2) {$\hsigma=$};
    \foreach \x/\val in {-11/3,-10/1,-9/2,-8/4,-7/6,-6/5} {
      \draw (\x,1) rectangle ++(1,-1);
        \node at (\x+0.5,0.5) {\val};
    }
    \foreach \x/\val in {-11/3,-10/1,-9/2,-8/4,-7/6,-6/5} {
      \draw (\x,2) rectangle ++(1,-1);
        \node at (\x+0.5,1.5) {\val};
    }
    \foreach \x/\val in {-11/3,-10/1,-9/2,-7/4,-6/6} {
      \draw (\x,3) rectangle ++(1,-1);
        \node at (\x+0.5,2.5) {\val};
    }
    
    \draw[thick, blue] (-11,3) rectangle ++(1,-1);
    \node at (-10.5,2.5) {\textcolor{blue}{$\mathbf{3}$}};

    \draw[thick, blue] (-7,3) rectangle ++(1,-1);
    \node at (-6.5,2.5) {\textcolor{blue}{$\mathbf{4}$}};
    \foreach \x/\val in {-11/5,-6/3} {
      \draw (\x,4) rectangle ++(1,-1);
        \node at (\x+0.5,3.5) {\val};
    }
    \draw[thick, blue] (-11,4) rectangle ++(1,-1);
    \node at (-10.5,3.5) {\textcolor{blue}{$\mathbf{5}$}};

    \node at (-4.5,2) {,};
  
    \foreach \x/\val in {0/1,1/3,2/2,3/6,4/4,5/5} {
      \draw (\x,1) rectangle ++(1,-1);
        \node at (\x+0.5,0.5) {\val};
    }
    \foreach \x/\val in {0/1,1/3,2/2,3/6,4/4,5/5} {
      \draw (\x,2) rectangle ++(1,-1);
        \node at (\x+0.5,1.5) {\val};
    }
    \foreach \x/\val in {0/1,1/3,2/2,4/6,5/4} {
      \draw (\x,3) rectangle ++(1,-1);
        \node at (\x+0.5,2.5) {\val};
    }
    \draw[thick, blue] (0,3) rectangle ++(1,-1);
    \node at (0.5,2.5) {\textcolor{blue}{$\mathbf{1}$}};
    
    \draw[thick, blue] (4,3) rectangle ++(1,-1);
    \node at (4.5,2.5) {\textcolor{blue}{$\mathbf{6}$}};
    \draw[thick, blue] (0,4) rectangle ++(1,-1);
    \node at (0.5,3.5) {\textcolor{blue}{$\mathbf{5}$}};

    \draw (5,4) rectangle ++(1,-1);
    \node at (5.5,3.5) {$1$};
    
    \node at (-2,2) {$f(\hsigma)=$};

    \draw[black, ultra thick] (-11,1) -- (-5,1);
    \draw[black, ultra thick] (0,1) -- (6,1);
  \end{scope}
\end{tikzpicture}
\end{center}
\end{example}

We are now ready to prove Theorem \ref{complement thm}.

\begin{proof}[Proof of Theorem \ref{complement thm}]
    Let $\hsigma\in\naf_\mu^{\pi_1\pi_2}$. It follows simply from the definition of $f$ that  each occurrence of a box with filling $1\leq i\leq m$ is in one-to-one correspondence with an occurrence of a box with filling $m+1-i$. Similarly, each occurrence of a box with filling $m+1\leq j\leq n$ is in one-to-one correspondence with an occurrence of a box with filling $n+m+1-i$. This is why we make the substitution of variables on the right side of \eqref{generalized equation}. By Lemma \ref{coinv lemma}, it follows that $\coinv(f(\hsigma))=\coinv'(\hsigma)$, which is why we make the variable substitution $t=t^{-1}$ on the right side of \eqref{generalized equation}. Similarly, by Lemma \ref{maj lemma}, $\maj(f(\hsigma))=\maj'(\hsigma)-p+L$. Therefore, we make the variable substitution $q=q^{-1}$ to account for $\maj'(\hsigma)$ and multiply by $q^{-p}$ to account for the $-p$ term. Additionally, to each variable $x_j$ for $m+1\leq j\leq n$, we attach the variable $q$ to account for the $L$ term, which follows by definition of $L$. Summing over all $\hsigma\in\naf_\mu^{\pi_1\pi_2}$ yields the desired result.
\end{proof}


\section{Recovering the partially symmetric identity}\label{CL section}

We use Theorem~\ref{complement thm} to prove the Concha-Lapointe identity of Theorem~\ref{CL theorem}, which we restate here for the reader's convenience:

\begin{theorem}[=\,Theorem~\ref{CL theorem}]\label{CL theorem text}
Suppose $n\in \mathbb{N}$ and $m$ is an integer such that $0\le m \le n$. For any partition $\lambda$ with at most $m$ parts and any $\gamma=(\gamma_1,\dotsc,\gamma_{n-m})\in(\Z_{\ge 0})^{n-m}$, we have
\begin{equation}\label{CL identity text}
P_{\lambda|\gamma}(x_1,\dotsc,x_m,qx_{n},\dotsc,qx_{m+1};q^{-1},t^{-1})
= q^{\gamma_1+\dotsm+\gamma_{n-m}}t^{\mathrm{inv}(\gamma)-\ell(\omega_0^{[m+1,n]})}T_{\omega_0^{[m+1,n]}}P_{\lambda|\gamma},
\end{equation}
where $\mathrm{inv}(\gamma)=|\{i<j : \gamma_i>\gamma_j\}|$.
\end{theorem}

In the setting of Theorem~\ref{complement thm}, we assume:
\begin{itemize}
\item $\mu=\lambda|\gamma$ where $\lambda\in(\mathbb{Z}_{\ge 0})^m$ with $\lambda_1\ge\dotsm\ge\lambda_m$ and $\gamma\in(\mathbb{Z}_{\ge 0})^{n-m}$,
\item $\pi_1$ belongs to $S_m=S_{[m]}\subset S_n$,
\item $\pi_2$ is the identity element of $S_{[m+1,n]}$.
\end{itemize}
In this case, $\ell_\mu(\pi_1)$ counts the number of pairs $(i,j)\in[m]^2$ such that $i<j, \pi_1(i)>\pi_1(j),$ and $\lambda_i=\lambda_j$.

Let $(S_m)_\lambda$ denote the stabilizer of $\lambda$, which is a standard parabolic subgroup of $S_m$. Let $S_m^\lambda$ be the set of minimal coset representatives for $S_m/(S_m)_\lambda$. For arbitrary $\pi_1\in S_m$, let $\underline{\pi_1}\in (S_m)_\lambda$ be the minimal representative of the coset $\pi_1(S_m)_\lambda$.

\begin{lemma}\label{E min lemma}
    With the notation and assumptions above, one has the following:
    \begin{enumerate}
        \item For $\pi_1\in S_m^\lambda$, $\ell_\mu(\pi_1)=0$ and hence $E_\mu^{\pi_1}=T_{\pi_1}E_\mu$.
        \item For $\pi\in S_m$, $E_\mu^{\pi_1}=E_\mu^{\underline{\pi_1}}$.
    \end{enumerate}
\end{lemma}

\begin{proof}
    Since $\pi_1\in S_m$ and $\lambda$ is weakly decreasing, $\ell_\mu(\pi_1)$ is equal to the number of pairs $(i,j)\in [m]^2$ such that $i<j$, $\pi(i)>\pi(j)$, and $\lambda_i=\lambda_j$. But membership in $S_m^\lambda$ is characterized by $\pi_1(i)<\pi_1(j)$ for all $(i,j)\in [m]^2$ such that $i<j$ and $\lambda_i=\lambda_j$. This establishes the first statement.

    For the second statement, let $\pi_1\in S_m$ be arbitrary and write $\pi_1=\underline{\pi_1}\sigma$ where $\sigma\in (S_m)_\lambda$. Then $\ell(\pi_1)=\ell(\underline{\pi_1})+\ell(\sigma)$ and hence $T_{\pi_1}=T_{\underline{\pi_1}}T_\sigma$. Since $\sigma$ can be expressed as a product of simple transpositions fixing $\lambda$, we have $T_\sigma E_\mu=t^{\ell(\sigma)}E_\mu$ by \eqref{Estab}. Hence
        $$
        E_\mu^{\pi_1} = t^{-\ell_\mu(\pi_1)}T_{\underline{\pi_1}}T_\sigma E_\mu = t^{-\ell_\mu(\pi_1)+
        \ell(\sigma)}T_{\underline{\pi_1}}E_\mu = t^{-\ell_\mu(\pi_1)+
        \ell(\sigma)}E_\mu^{\underline{\pi_1}}.
        $$
    The factorization $\pi_1=\underline{\pi_1}\sigma$ gives that for all $(i,j)\in[m]^2$ such that $i<j$ and $\lambda_i=\lambda_j$, $\pi_1(i)>\pi_1(j)$ if and only if $\sigma(i)>\sigma(j)$. Thus $\ell_\mu(\pi_1)=\ell(\sigma)$ and we are done. 
\end{proof}


\begin{proof}[Proof of Theorem~\ref{CL theorem text}]

By the definition of $P_{\lambda|\gamma}$ and Lemma~\ref{E min lemma}, we have 
\begin{equation*}
    P_{\lambda|\gamma} = \sum_{\pi_1\in S_m^\lambda} T_{\pi_1}E_\mu = \sum_{\pi_1\in S_m^\lambda} E_\mu^{\pi_1},
\end{equation*}
and by the definition of $E^\pi_\mu$ it follows that
\begin{equation*}
    T_{\omega_0^{[m+1,n]}}P_{\lambda|\gamma} = \sum_{\pi_1\in S_m^\lambda} t^{\ell_\mu(\omega_0^{[m+1,n]})}E^{\pi_1\omega_0^{[m+1,n]}}_{\mu}.
\end{equation*}

Recall that $p=\sum_{i=m+1}^n\mu_i=\sum_{i=1}^{n-m}\gamma_i$. We now invoke Theorem~\ref{complement thm} to obtain
\begin{align*}
    q^p T_{\omega_0^{[m+1,n]}}P_{\lambda|\gamma} 
    &= \sum_{\pi_1\in S_m^\lambda} q^p t^{\ell_\mu(\omega_0^{[m+1,n]})}E^{\pi_1\omega_0^{[m+1,n]}}_{\mu}(x_1,\dotsc,x_n;q,t)\\
    &= \sum_{\pi_1\in S_m^\lambda} t^{\ell_\mu(\omega_0^{[m+1,n]})}E^{\pi_1^c}_{\mu}(x_m,\dotsc,x_1,qx_n,\dotsc,qx_{m+1};q^{-1},t^{-1})\\
    &= \sum_{\pi_1\in S_m^\lambda} t^{\ell_\mu(\omega_0^{[m+1,n]})}E^{\underline{\pi_1^c}}_{\mu}(x_m,\dotsc,x_1,qx_n,\dotsc,qx_{m+1};q^{-1},t^{-1}),
\end{align*}
where we also use Lemma~\ref{E min lemma} for the last equality.

The quantity $\ell_\mu(\omega_0^{[m+1,n]})$ is the number of pairs $(i,j)\in[m+1,n]^2$ such that $i<j$ and $\mu_i\le\mu_j$, which is exactly $\ell(\omega_0^{[m+1,n]})-\text{inv}(\gamma)$. Therefore
\begin{align*}
    q^p T_{\omega_0^{[m+1,n]}}P_{\lambda|\gamma} 
    &= t^{\ell(\omega_0^{[m+1,n]})-\text{inv}(\gamma)}\sum_{\pi_1\in S_m^\lambda} E^{\underline{\pi_1^c}}_\mu(x_m,\dotsc,x_1,qx_n,\dotsc,qx_{m+1};q^{-1},t^{-1})\\
    &= t^{\ell(\omega_0^{[m+1,n]})-\text{inv}(\gamma)} P_{\lambda|\gamma}(x_1,\dotsc,x_m,qx_n,\dotsc,qx_{m+1};q^{-1},t^{-1}),
\end{align*}
where we use the fact that $\pi_1\mapsto\underline{\pi_1^c}$ is a bijection from $S_m^\lambda$ to itself and, finally, the symmetry of $P_{\lambda|\gamma}$ in $x_1,\dotsc,x_m$.
\end{proof}

\section{Kazhdan-Lusztig involution}

Let $\mathcal{M}=\mathbb{Q}(q,t^{1/2})[x_1,\dotsc,x_n]$, where $t^{1/2}$ is a formal variable such that $(t^{1/2})^2=t$. As explained (in greater generality) in \cite{I}, the space $\mathcal{M}$ has the structure of \textit{maximal parabolic module} for the (positive) affine Hecke algebra $\mathcal{H}_n$. This action is generated by the Demazure-Lusztig operators $T_i$ for $i=1\dotsc,n-1$ and the operators of multiplication by $x_i$ for $i=1,\dotsc,n$. Moreover, the \textit{Kazhdan-Lusztig involution} $*$ on $\mathcal{H}_n$ induces the following $\mathbb{Q}$-linear involution on $\mathcal{M}$:
\begin{equation*}
f\longmapsto f^*=t^{\ell(\omega_0)}T_{\omega_0}^{-1}\omega_0(f(q^{-1},t^{-1})),
\end{equation*}
where $\omega_0=\omega_0^{[n]}$ is the longest element of $S_n$ and $f(q^{-1},t^{-1})$ denotes the image of $f$ under the $\mathbb{Q}$-algebra automorphism of $\mathcal{M}$ sending $q\mapsto q^{-1}, t^{1/2}\mapsto t^{-1/2}$, and $x_i\mapsto x_i$ for $i=1,\dotsc,n$. By definition, we have
\begin{equation*}
T_\sigma^* = T_{\sigma^{-1}}^{-1}
\end{equation*}
for all $\sigma \in S_n$.

\subsection{Non-symmetric case}

The following result of Ion \cite{I} describes the action of the Kazhdan-Lusztig involution on nonsymmetric Macdonald polynomials.

\begin{theorem}[{\cite[Theorem 4.8]{I}}]\label{I theorem}
For any $\mu=(\mu_1,\dotsc,\mu_n)\in(\Z_{\ge 0})^n$, the normalized non-symmetric Macdonald polynomial $\tilde{E}_\mu=t^{\frac{\mathrm{inv}(\mu)}{2}}E_\mu$ satisfies $\tilde{E}_\mu^*=\tilde{E}_\mu$.
\end{theorem}

We observe that $\tilde{E}_\mu^*=\tilde{E}_\mu$ is equivalent to the non-symmetric case ($m=0$) of the Concha-Lapointe identity \eqref{CL identity}, namely:
\begin{equation}\label{m0equiv}
E_\mu(x_{n},\dotsc,x_{1};q^{-1},t^{-1})
= t^{\mathrm{inv}(\mu)-\ell(\omega_0)}T_{\omega_0}E_\mu(x;q,t).
\end{equation}
Here we use the fact that $E_\mu$ is homogeneous of degree $\mu_1+\dotsm+\mu_n$ to eliminate the $q$-shifts.

\subsection{Partially symmetric case}

For $0\le m\le n$, define
\begin{align*}
e_m = t^{-\frac{\ell(\omega_0^{[m]})}{2}}\sum_{\sigma\in S_m} T_\sigma.
\end{align*}
It is well-known and straightforward to verify that $e_m^* = e_m$ and also that 
\begin{equation*}
e_m\mathcal{M}=\mathbb{Q}(q,t)[x_1,\dotsc,x_n]^{S_m}.
\end{equation*}
Therefore, we may restrict the Kazhdan-Lusztig involution to $e_m\mathcal{M}$, where it is given by:
\begin{equation}\label{KL eM}
(e_mf)^*=e_m^*f^*=e_mt^{\ell(\omega_0)}T_{\omega_0}^{-1}\omega_0(f(q^{-1},t^{-1})).
\end{equation}

Now we are ready to prove Theorem~\ref{KL theorem}, which gives the partially symmetric extension of Theorem~\ref{I theorem} and which we restate here for the reader's convenience.

\begin{theorem}[=\,Theorem~\ref{KL theorem}]\label{KL theorem text}
Let $n \in \mathbb{N}$ and $0\le m \le n$. For any partition $\lambda$ with at most $m$ parts and any $\gamma=(\gamma_1,\dotsc,\gamma_{n-m})\in(\Z_{\ge 0})^{n-m}$,
\begin{equation}\label{KL P text}
P_{\lambda|\gamma}^* = t^{\mathrm{inv}(\lambda|\gamma)}P_{\lambda|\gamma},
\end{equation}
where $\mathrm{inv}(\lambda|\gamma)=\mathrm{inv}(\gamma)+|\{(i,j)\in[m]\times[n-m]:\lambda_i>\gamma_j\}|$. Furthermore, equations \eqref{CL identity text} and \eqref{KL P text} are equivalent.
\end{theorem}

\begin{proof}
For $\mu=\lambda|\gamma$ with $\lambda\in(\mathbb{Z}_{\ge 0})^m$ weakly decreasing and $\gamma\in(\mathbb{Z}_{\ge 0})^{n-m}$, we have
\begin{align*}
e_m \tilde{E}_\mu 
&= t^{\frac{\inv(\mu)-\ell(\omega_0^{[m]})}{2}}\sum_{\sigma\in S_m}T_\sigma E_\mu\\
&= t^{\frac{\inv(\mu)-\ell(\omega_0^{[m]})}{2}}\left(\sum_{\sigma\in (S_m)_\lambda}t^{\ell(\sigma)}\right) P_{\lambda|\gamma}.
\end{align*}
Since $(e_m \tilde{E}_\mu)^* = e_m^*\tilde{E}_\mu^* = e_m \tilde{E}_\mu$ by Theorem~\ref{I theorem}, we obtain
\begin{equation}\label{P star 0}
t^{\frac{-\inv(\mu)+\ell(\omega_0^{[m]})}{2}}\left(\sum_{\sigma\in (S_m)_\lambda}t^{-\ell(\sigma)}\right) P_{\lambda|\gamma}^* = t^{\frac{\inv(\mu)-\ell(\omega_0^{[m]})}{2}}\left(\sum_{\sigma\in (S_m)_\lambda}t^{\ell(\sigma)}\right) P_{\lambda|\gamma}.
\end{equation}
Denoting by $\omega_{0,\lambda}^{[m]}$ the longest element of the stabilizer $(S_m)_\lambda$, we have the identity
\begin{equation*}
\sum_{\sigma\in (S_m)_\lambda}t^{-\ell(\sigma)} = t^{-\ell(\omega_{0,\lambda}^{[m]})}\sum_{\sigma\in (S_m)_\lambda}t^{\ell(\sigma)},
\end{equation*}
and hence \eqref{P star 0} reduces to
\begin{align}\label{P star 1}
P_{\lambda|\gamma}^*
= t^{\inv(\mu)+\ell(\omega_{0,\lambda}^{[m]})-\ell(\omega_0^{[m]})}P_{\lambda|\gamma}.
\end{align}
Since $\lambda$ is weakly decreasing, we have
\begin{equation*}
\text{inv}(\mu)+\ell(\omega_{0,\lambda}^{[m]})-\ell(\omega_0^{[m]})=\text{inv}(\gamma)+|\{(i,j)\in[m]\times[n-m]:\lambda_i>\gamma_j\}|
\end{equation*}
and this proves the first assertion of Theorem~\ref{KL theorem text}.

To connect with the Concha-Lapointe identity \eqref{CL identity text} we need to work out the left-hand side of \eqref{P star 1}. Using \eqref{KL eM} for the Kazhdan-Lusztig involution, we may write this as
\begin{align}\label{P star 2}
P_{\lambda|\gamma}^* &= \frac{t^{-\frac{\ell(\omega_0^{[m]})}{2}}}{\sum_{\sigma\in (S_m)_\lambda}t^{-\ell(\sigma)}}t^{\ell(\omega_0)} e_m T_{\omega_0}^{-1}\omega_0 (E_\mu(q^{-1},t^{-1})).
\end{align}
In order to pass from $T_{\omega_0}^{-1}$ in this formula to $T_{\omega_0^{[m+1,n]}}^{-1}$ in the Concha-Lapointe identity, we make use of the following eigenoperator for non-symmetric Macdonald polynomials (cf. \cite[Prop. 3.3.1]{C:book}):
\begin{equation*}
Y_{-\varpi_m} = t^{-m(m-1)/2}(T_{m}^{-1}\dotsm T_{n-1}^{-1})\dotsm(T_2^{-1}\dotsm T_{n-m+1}^{-1})(T_{1}^{-1}\dotsm T_{n-m}^{-1})\pi^{-m},
\end{equation*}
where $\varpi_m$ stands for the $m$-th $GL_n$-fundamental weight and $\pi^{-1}$ is the operator
\begin{equation*}
\pi^{-1}f=f(qx_n,x_1,\dotsc,x_{n-1}).
\end{equation*}
Here we need to assume that $m>0$, and there is no harm in doing so, since the equivalence has already been observed in the $m=0$ case in \eqref{m0equiv}.

The eigenvalues of $Y_{-\varpi_m}$ on $E_\mu$ for $\mu=\lambda|\gamma$ are given as follows:
\begin{equation}\label{YE}
Y_{-\varpi_m}E_\mu = q^{\lambda_1+\dotsm+\lambda_m}t^{-\sum_{i=1}^m b_\mu(i)} E_\mu,
\end{equation}
where $b_\mu(i)$ is equal to
\begin{align}\label{eigenvalues}
|\{j\in[m]:j<i,\lambda_j>\lambda_i\}|+\{j\in[m]:j>i,\lambda_j=\lambda_i\}+\{j\in[n-m]:\gamma_j\ge\lambda_i\}|.
\end{align}

Using \eqref{YE}, we can write \eqref{P star 2} as
\begin{equation}\label{P star 3}
P_{\lambda|\gamma}^*
=\frac{t^{-\frac{\ell(\omega_0^{[m]})}{2}}}{\sum_{\sigma\in (S_m)_\lambda}t^{-\ell(\sigma)}}t^{\ell(\omega_0)-\sum_{i=1}^m b_\mu(i)}
q^{\lambda_1+\dotsm+\lambda_m}e_mT_{\omega_0}^{-1}\omega_0((Y_{-\varpi_m}E_\mu)(q^{-1},t^{-1})).
\end{equation}
It is straightforward to verify that for all $i=1,\dotsc, n-1$,
\begin{equation}\label{T star}
\omega_0(T_i^{-1}f)(q^{-1},t^{-1}) = T_{n-i}\,\omega_0(f(q^{-1},t^{-1})),
\end{equation}
and hence
\begin{align*}
&T_{\omega_0}^{-1} \omega_0 ((Y_{-\varpi_m}f)(q^{-1},t^{-1}))\\
&= t^{m(m-1)/2}T_{\omega_0}^{-1}(T_{n-m}\dotsm T_{1})\dotsm(T_{n-2}\dotsm T_{m-1})(T_{n-1}\dotsm T_{m})\omega_0((\pi^{-m}f)(q^{-1},t^{-1}))\\
&= t^{m(m-1)/2} T_{\omega_0^{[m]}}^{-1}T_{\omega_0^{[m+1,n]}}^{-1}\omega_0((\pi^{-m}f)(q^{-1},t^{-1}))\\
&= t^{m(m-1)/2} T_{\omega_0^{[m]}}^{-1}T_{\omega_0^{[m+1,n]}}^{-1}(f(q^{-1}x_m,\dotsc,q^{-1}x_1,x_n\dotsc,x_{m+1};q^{-1},t^{-1})),
\end{align*}
for all $f=f(q,t)=f(x_1,\dotsc,x_n;q,t)\in\mathcal{M}$. In the second equality above, we use the reduced decomposition
\begin{equation*}
\omega_0 = \omega_0^{[m]}\omega_0^{[m+1,n]}(\sigma_{m}\dotsm \sigma_{n-1})(\sigma_{m-1}\dotsm \sigma_{n-2})\dotsm(\sigma_1\dotsm \sigma_{n-m}),
\end{equation*}
together with the fact that the $T_i^{-1}$ for $i=1,\dotsc,n-1$ satisfy the braid relations.

Returning to \eqref{P star 3} and using
\begin{align*}
e_mT_{\omega_0^{[m]}}^{-1}=t^{-\ell(\omega_0^{[m]})}e_m=t^{-m(m-1)/2}e_m,\qquad e_mT_{\omega_0^{[m+1,n]}}^{-1}=T_{\omega_0^{[m+1,n]}}^{-1}e_m,
\end{align*}
we can now write
\begin{align}
P_{\lambda|\gamma}^* &= \frac{t^{-\frac{\ell(\omega_0^{[m]})}{2}}}{\sum_{\sigma\in (S_m)_\lambda}t^{-\ell(\sigma)}}t^{\ell(\omega_0)-\sum_{i=1}^m b_\mu(i)} q^{\lambda_1+\dotsm+\lambda_m}\times\notag\\
&\qquad\qquad\qquad\qquad T_{\omega_0^{[m+1,n]}}^{-1}e_mE_\mu(q^{-1}x_{m},\dotsc,q^{-1}x_1,x_n,\dotsc,x_{m+1};q^{-1},t^{-1}).\label{P star 4}
\end{align}
Since $E_\mu$ is homogeneous of degree $\sum_{i=1}^m\lambda_i+\sum_{i=1}^{n-m}\gamma_i$ (and recall that $p=\sum_{i=1}^{n-m}\gamma_i$), the right-hand side of \eqref{P star 4} is equal to
\begin{align*}
&\frac{t^{-\frac{\ell(\omega_0^{[m]})}{2}}}{\sum_{\sigma\in (S_m)_\lambda}t^{-\ell(\sigma)}}t^{\ell(\omega_0)-\sum_{i=1}^m b_\mu(i)} q^{-p}\times\\
&\qquad\qquad\qquad\qquad T_{\omega_0^{[m+1,n]}}^{-1}e_mE_\mu(x_{m},\dotsc,x_1,qx_{n},\dotsc,qx_{m+1};q^{-1},t^{-1}).
\end{align*}
Now we factor out the action of permutations on the variables to write
\begin{align*}
&e_mE_\mu(x_{m},\dotsc,x_1,qx_{n},\dotsc,qx_{m+1};q^{-1},t^{-1})\\
&\qquad\qquad\qquad\qquad =\omega_0^{[m+1,n]}e_m\omega_0^{[m]}E_\mu(x_{1},\dotsc,x_m,qx_{m+1},\dotsc,qx_n;q^{-1},t^{-1}),
\end{align*}
and we repeatedly apply \eqref{T star} with $i=1,\dotsc,m-1$ and $\omega_0^{[m]}$ in place of $\omega_0$ to write
\begin{align*}
&e_m\omega_0^{[m]}E_\mu(x_{1},\dotsc,x_m,qx_{m+1},\dotsc,qx_n;q^{-1},t^{-1})\\
&\qquad\qquad\qquad\qquad = \omega_0^{[m]}(e_m^*E_\mu)(x_{1},\dotsc,x_m,qx_{m+1},\dotsc,qx_n;q^{-1},t^{-1})\\
&\qquad\qquad\qquad\qquad = \omega_0^{[m]}(e_m E_\mu)(x_{1},\dotsc,x_m,qx_{m+1},\dotsc,qx_n;q^{-1},t^{-1}).
\end{align*}
Putting these steps together, we can write \eqref{P star 4} as
\begin{align*}
P_{\lambda|\gamma}^* &= t^{\ell(\omega_0)-\sum_{i=1}^m b_\mu(i)} q^{-p}T_{\omega_0^{[m+1,n]}}^{-1}\omega_0^{[m+1,n]}P_{\lambda|\gamma}(x_1,\dotsc,x_m,qx_{m+1},\dotsc,qx_n;q^{-1},t^{-1}).
\end{align*}

Finally, we see that the equivalence of \eqref{CL identity text} and \eqref{KL P text} boils down to the equation
\begin{equation*}
\ell(\omega_0)-\sum_{i=1}^mb_\mu(i)-\text{inv}(\gamma)-|\{(i,j)\in[m]\times[n-m]:\lambda_i>\gamma_j\}|=\ell(\omega_0^{[m+1,n]})-\text{inv}(\gamma).
\end{equation*}
To obtain the desired quantity $\ell(\omega_0^{[m+1,n]})-\text{inv}(\gamma)$ from the left-hand side, we substitute \eqref{eigenvalues} for each $b_\mu(i)$ and then cancel from $\ell(\omega_0)$ all pairs $(i,j)\in[n]^2$ with $i<j$ except those with $i,j\in[m+1,n]$.
\end{proof}

\bibliographystyle{plain}
\bibliography{bibliography}

\end{document}